\documentclass{amsart}

\usepackage{graphicx, color}

\def\blue{\color{black}}

\newtheorem{thm}{Theorem}[section]
\newtheorem{cor}[thm]{Corollary}
\newtheorem{lem}[thm]{Lemma}
\newtheorem{prop}[thm]{Proposition}
\theoremstyle{definition}

\theoremstyle{remark}
\newtheorem{rem}[thm]{Remark}
\numberwithin{equation}{section}
\begin{document}

\title[]
 {Existence and applications of Ricci flows via pseudolocality}

\author{Fei He}

\address{School of Mathematical Science, Xiamen University}

\email{hefei@xmu.edu.cn}

\thanks{}

\thanks{}



\date{}

\dedicatory{}

\commby{}

% -----------------------------------------------------------------------------

\begin{abstract}
We prove the short-time existence of Ricci flows on complete manifolds with scalar curvature bounded below uniformly, Ricci curvature bounded below by a negative quadratic function, and with almost Euclidean isoperimetric inequality holds locally. In particular, this result applies to manifolds with both Ricci curvature and injectivity radius bounded from below. We also study the short-time behaviour of these solutions which may have unbounded curvature at the initial time, and provide some applications. A key tool is Perelman's pseudolocality theorem.
\end{abstract}

% -----------------------------------------------------------------------------
\maketitle
% -----------------------------------------------------------------------------

\section{Introduction}

Since its introduction by Hamilton \cite{hamilton1982three}, the Ricci flow has been extensively studied and very fruitful in yielding deep results, although there are still many questions to be answered. One of the many open problems about the Ricci flow is the existence of solutions on noncompact manifolds. In fact, existence of the flow on general complete manifolds is expected to be not true, and we would like to restrict our attention to manifolds with curvature conditions relevant to potential applications. When the sectional curvature is uniformly bounded, it follows from the work of Shi \cite{shi1989deforming} that the Ricci flow has a unique solution. When the curvature is unbounded, progress has been made by many authors under various assumptions, for an incomplete list, see Cabezas-Rivas and Wilking \cite{cabezas2011produce}, Chau, Li and Tam \cite{chau2014deforming}, Giesen and Topping \cite{giesen2011existence}\cite{giesen2013ricci}, Hochard \cite{hochard2016short}, Simon \cite{simon2002deformation}\cite{simon2012ricci}, Topping \cite{topping2010ricci}, Xu \cite{xu2013short}, and the author of this article \cite{he2016existence}. {\blue For more recent progress, one can see for example  \cite{Bamler2019Ricci} \cite{Lai2019Ricci} \cite{lee2017chern}and \cite{Chau2020The}.}
In this article, we study the existence of Ricci flows on complete noncompact manifolds with a set of assumptions motivated by Perelman's pseudolocality theorem. We focus on dimension $n\geq 3$ since the $2$-dimensional case has been settled in \cite{topping2010ricci} and \cite{giesen2011existence}, which also make use of pseudolocality. We also explore the properties of these solutions and point out some applications.

It is well-known that for any regular domain $\Omega$ in the Euclidean space $\mathbb{R}^n$, we have the \emph{Euclidean isoperimetric inequality}
\[ |\partial \Omega|^n \geq I_n |\Omega|^{n-1}, \]
where $I_n = \frac{|\partial B_{\mathbb{R}^n}(0,1)|^n}{|B_{\mathbb{R}^n}(0,1)|^{n-1}}$ is the optimal constant.
On a Riemannian manifold $(M^n,g)$, let $B_g(p,r)$ be a geodesic ball with radius $r$ centered at a point $p$. We say the \emph{$\delta-$almost Euclidean isoperimetric inequality} holds in $B_g(p,r)$ with respect to the Riemannian metric $g$, if
\[(Area_g(\partial \Omega))^n \geq (1-\delta)I_n (Vol_g(\Omega))^{n-1}\]
for any regular domain $\Omega \subset B_g(p,r)$. In \cite{perelman2002entropy}, Perelman proved an interior curvature estimate for compact Ricci flows known as the pseudolocality theorem. The complete noncompact case has been verified by Chau, Tam and Yu \cite{chau2779131pseudolocality}. See also \cite{chow2010ricci} for a detailed treatment. {\blue There is another version of pseudolocality estimate for the Ricci flow proved by Tian-Wang \cite{Tian2015On}, where almost nonnegative Ricci curvature is assumed, while the almost Euclidean isoperimetric inequality assumption is relaxed to almost Euclidean volume. }

\begin{thm}[Perelman's pseudolocality]\label{thm: pseudolocality}
For every $n$ and $A > 0$, there exist $\delta_0 >0$ and $\epsilon_0 >0$ depending only on $A $ and $n$ with the following property: Suppose $(M^n, g(t)), t\in [0, (\epsilon r)^2]$ is a complete solution of the Ricci flow with bounded curvature, where $0< \epsilon \leq \epsilon_0$ and $r>0$. Let $x_0$ be a point in $M$. If there is a scalar curvature lower bound
\[S(x,0) \geq -r^{-2} \quad for \quad any \quad x \in B_{g(0)}(x_0, r),\]
and if the $\delta_0-$almost Euclidean isoperimetric inequality holds in $B_{g(0)}(x_0,r)$ with respect to the initial metric $g(0)$, then we have
\[|Rm|(x,t) \leq \frac{A}{t}+\frac{1}{(\epsilon_0 r)^2}\]
for $x\in B_{g(t)}(x_0, \epsilon_0 r)$ and $t\in (0, (\epsilon r)^2]$.
\end{thm}
The validity of the $\delta-$almost Euclidean isoperimetric inequality under a fixed radius is actually a strong condition, roughly speaking it rules out too much positive curvature. Nevertheless, it does not require any point-wise curvature upper bound. An example is given by the neighbourhood of a rounded out flat cone point, with cone angle close to $2\pi$ so it is $C^0-$close to a Euclidean disk, while the sectional curvature can be made arbitrarily large and positive, and the injectivity radius can be arbitrarily small.

On the other hand, bounded curvature and a volume lower bound for all unit geodesic balls can imply the validity of the $\delta-$almost Euclidean isoperimetric inequality. By the work of Anderson and Cheeger \cite{anderson1992c}, a lower bound of both the Ricci curvature and the injectivity radius implies a lower bound of the $W^{1,p}-$harmonic radius, hence also implies this condition. {\blue A recent result of Cavalletti and Mondino shows that a Ricci lower bound together with almost optimal volume for unit balls imply the almost Euclidean isoperimetric inequality \cite{cavalletti2017almost}.}

Perelman's pseudolocality is not really a local result since it assumes a complete Ricci flow with bounded curvature, the completeness is necessary. Nevertheless, we can apply it to prove the short-time existence of Ricci flow solutions, with possibly unbounded curvature at the initial time. The key step is a conformal transformation which turns a compact domain into a complete Riemannian manifold, while keeping the scalar curvature lower bound and the isoperimetric inequality.
\begin{thm}\label{thm: short-time existence}
For any $n>0$, $A > 0$, $k\geq 0$ and $L\geq 0$, there exists constants $\delta_1>0$ and $\epsilon_1$ depending only on $n$ and $A$, such that if $(M^n,g)$ is a complete Riemannian manifold satisfying

 (i) $\liminf_{d(p,x)\to \infty} \frac{Ric(x)}{d(p,x)^2} \geq - L$, where $d(p,x)$ is the geodesic distance function from a fixed point $p\in M$,

  (ii) scalar curvature $S\geq -k$,

  (iii) the $\delta_1-$almost Euclidean isoperimetric inequality holds in any geodesic ball with radius $r>0$,

then $(M,g)$ admits a complete solution of the Ricci flow with curvature bound
\[|Rm(x,t)| \leq \frac{A}{t} + \frac{1}{(\epsilon_1 \bar{r})^2},\]
for all $x\in M$ and $0<t\leq (\epsilon_1 \bar{r})^2$, where {\blue $\bar{r}=\min\{r, \sqrt{1/(L+1)k}\}$,  this solution may also depend on $r$}. Moreover, $g(t)$ is $\kappa(n)-$noncollapsed under the scale $\sqrt{t}$, for any $0<t\leq (\epsilon_1 \bar{r})^2$.
 %$\kappa(n)-$noncollapsed under the scale $\sqrt{t}$ for all $t\in(0, (\epsilon_1\bar{r})^2]$, where $\kappa(n)>0$ is a constant depending only on $n$.
\end{thm}

Recall that we say a Riemannian manifold $(M,g)$ is $\kappa-$noncollapsed under the scale $\rho$ if
\[\frac{Vol_gB_g(r)}{r^n} \geq \kappa, \quad for \quad all \quad 0< r< \rho.\]
When the initial manifold has bounded curvature, the metrics under a smooth Ricci flow solution are uniformly equivalent to the initial metric, so there is no worry about drastic distance distortion or volume collapsing within a short time. This is no longer clear when the curvature is only bounded by a non-integrable function $\frac{A}{t}$. However, the solutions from Theorem \ref{thm: short-time existence} are noncollapsed for any positive time close enough to $0$, hence enjoy the compactness property on positive time intervals in the sense of Hamilton \cite{hamilton1995compactness}.

{\blue
As an interesting special case, we obtain the following corollary from the above theorem and \cite{cavalletti2017almost}.
\begin{cor}\label{cor: Ricci and injectivity radius bounded below}
For any dimenion $n\geq 2$, any constants $A>0$, $\epsilon > 0$, there exist constants $\eta(n, A, \epsilon)$ and $\tau(n, A, \epsilon)$, such that the following holds. Let $(M,g)$ be a complete $n$-dimension smooth Riemannian manifold with the following conditions:

(1) there is a constant $l \geq 0$, such that $Ric \geq -l g$,

(2) $Vol_g B_g(x, r) \geq (1-\eta) \omega_n r^n$ for any $r \in [0, 1]$.

\noindent 
Then there is a complete Ricci flow solution $(M,g(t))$, $t\in [0,\tau/l]$, with $g(0)=g$, satisfying properties:

(i) $|Rm|(x, t) \leq \frac{A}{t} + \frac{l}{\tau}$;

(ii) $Vol_{g(t)}B_{g(t)}(x, \sqrt{t}) \geq (1-\epsilon)\omega_n(\sqrt{t})^n$;

\noindent where $x\in M$ and $t\in (0,\tau/l]$.
\end{cor}
%\begin{cor}\label{cor: Ricci and injectivity radius bounded below}
%Let $(M^n,g)$ be a complete Riemannian manifold. Suppose there are constants $l$ and $\iota>0$, such that the Ricci curvature and injectivity radius are bounded from below uniformly
%\[ Ric(x) \geq -l \quad and \quad inj_x \geq \iota,\]
%then there is a complete smooth Ricci flow solution with $(M^n,g)$ as initial data, moreover, this solution is noncollapsed in the sense of Proposition \ref{prop: short-time noncollapsing}.
%\end{cor}
\begin{rem}
It has been proved in \cite{hochard2016short} and \cite{Simon2016Local} that in dimension $3$, the almost Euclidean volume condition can be weakened to volume noncollapsing. 
\end{rem}
\begin{rem}\label{rem: injectivity radius lower bound}
By a theorem of Anderson-Cheeger \cite{anderson1992c}, Ricci lower bound and positive injectivity radius lower bound imply a lower bound for the harmonic radius, hence the almost Euclidean isoperimetric inequality (and alomst Euclidean volume growth) holds within a certain radius, therefore the above theorem can be applied in this case.  
\end{rem}
}

We would like to list some direct applications of our existence result. First recall that a curvature lower bound and a $C^0$ (almost) optimal condition can imply rigidity of Riemannian manifolds. For example, consider a complete Riemannian manifold $(M^n,g)$ with nonnegative Ricci curvature, it follows directly from volume comparison theorem that if the volume growth rate is (exactly) Euclidean, then $(M^n,g)$ is isometric to the Euclidean space $\mathbb{R}^n$; if the volume growth is almost optimal, Cheeger and Colding showed in \cite{cheeger1997structure} that the manifold is diffeomorphic to $\mathbb{R}^n$. Our first corollary is in the same spirit.
\begin{cor}\label{cor: rigidity scalar nonnegative}
Let $(M^n,g)$ be a complete Riemannian manifold. Suppose

(i) $\liminf_{d(p,x) \to \infty }\frac{Ric(x)}{d(p,x)^2} > -\infty$,

(ii) the scalar curvature $S \geq 0$,

(iii) the optimal Euclidean isoperimetric inequality holds on $M$.

Then $(M,g)$ is isometric to the Euclidean space $\mathbb{R}^n$.
\end{cor}
%\begin{cor}
%For any $n$ and $\rho > 0$ there is a constant $\delta(n,\rho)>0$, such that if $(M^n,g)$ is a complete Riemannian manifold satisfying
%
%(i) $\liminf_{d(p,x) \to \infty }\frac{Ric(x)}{d(p,x)^2} > -\infty$,
%
%(ii) the scalar curvature $S \geq 0$,
%
%(iii) the $\delta(n, \rho)-$almost Euclidean isoperimetric inequality holds on $M$,
%
%\noindent then any geodesic ball of radius $\rho$ in $M^n$ is diffeomorphic to a Euclidean ball of dimension $n$.
%
%\end{cor}
Ricci curvature lower bound and the validity of isoperimetric inequalities does not imply injectivity lower bound, however we can show that it rules out nontrivial topology locally under a certain scale.
\begin{cor}\label{cor: local topology control}
For any $n$ and $k$, {\blue there} exist constants $\delta$ and $\eta$ depending only on $n$ and $k$, such that if the geodesic ball $B(p,1)$ is relatively compact in a Riemmannian manifold $(M,g)$, $Ric \geq -k g$ on $B(p,1)$, and the $\delta-$almost Euclidean isoperimetric inequality holds in $B(p,1)$, then $B(p,\eta)$ is diffeomorphic to a Euclidean ball of dimension $n$.
\end{cor}

{
\blue 

Yau's uniformization conjecture predicts that complete K{\"a}hler manifolds with positive holomorphic bisectional curvature are biholomorphic to $\mathbb{C}^n$. This conjecture has generated a lot of research, especially in the study of complete noncompact Ricci flow, and there are many partial confirmations, one can refer to \cite{chau2007survey} for a survey, and to \cite{liu2016gromov} for the most recent progress. By \cite{chau2006complex}, Yau's conjecture is true under the additional assumptions of maximal volume growth and bounded curvature. The bounded curvature condition can be relaxed in certain situations as shown in \cite{huang2015k} using Ricci flow. 
%It has been proved in \cite{huang2015k} that the nonnegativity of holomorphic bisectional curvature is preserved under a smooth Ricci flow solution with curvature bounded by $\frac{A}{t}$, as long as $A$ is sufficiently small. 
Recently G. Liu proved in \cite{liu2016gromov} that this conjecture is true if only assume maximal volume growth in addition, his approach is different from the above mentioned work and does not rely on the Ricci flow. 
Shortly after, Lee and Tam \cite{lee2017chern} gave another proof of Liu\rq{}s result using the Ricci flow, in particular, a critical step in their proof is to establish the short-time existence of K\"ahler Ricci flow on noncollapsed K\"ahler manifolds with nonnegative bisectional curvature, but with potentially unbounded curvature.

Since nonnegative holomorphic bisectional curvature implies nonnegative Ricci curvature, Theorem \ref{thm: short-time existence} can be applied to reprove a special case of Liu\rq{}s result, see Corollary \ref{cor: nonnegative BK}. 

%\begin{cor}\label{cor: application to uniformization}
%For any $n$, there is $\delta(n)>0$, such that if $(M,g)$ is a complete $2n-$dimensional K\"ahler manifold with nonnegative holomorphic bisectional curvature and maximal volume growth, and suppose the $\delta(n)-$almost Euclidean isoperimetric inequality holds in every unit geodesic ball, then $(M,g)$ is biholomorphic to $\mathbb{C}^n$.
%\end{cor}

%\begin{rem}
%After the posting of the first draft of this work on arXiv, there has been many important progress on related problems. We are obliged to list some (but not all) recent results by other authors to keep the information up to date:

%i) Yau's uniformization conjecture in the maximal volume growth case was proved first by Liu in an updated version of \cite{liu2016gromov}. Later, Lee and Tam \cite{lee2017chern} gave another proof using Ricci flow, in particular, they proved the short-time existence of K\"ahler Ricci flow on noncollapsed K\"ahler manifolds with bisectional curvature bounded from below.

%ii) By a recent work of Cavalletti and Mondino \cite{cavalletti2017almost}, almost Euclidean isoperimetric inequality can be implied by almost Euclidean volume growth when there is a Ricci lower bound, hence the assumption in our Corollary \ref{cor: Ricci and injectivity radius bounded below} and \ref{cor: local topology control} can be weakened.
%\end{rem}
}

\textbf{Acknowledgement:} The author is grateful to the University of Minnesota, Twin Cities, where the first draft of this work was finished. He would like to thank Professor Jiaping Wang for helpful conversations.
{\blue And he is indebted to anonymous referees for providing very helpful comments.}

\section{Proof of existence}

Perelman's pseudolocality is applied to give the following estimate of the lifespan of Ricci flow solutions.
\begin{lem}\label{lem:lifespan via pseudolocality}
Let $\delta_0$ and $\epsilon_0$ be the constants in Theorem \ref{thm: pseudolocality} (where $A>0$ can be arbitrarily chosen). Let $(M^n,g(t)), t\in[0,T)$ be a complete Ricci flow solution, { \blue where $T>0$, such that $g(t)$ has bounded curvature for each $t\in[0,T)$, and 
%\[T = \inf \{\tau: \limsup_{t\to \tau^{-}} sup_M |Rm|_{g(t)} = \infty\} \cup \{\infty\}.\]
\[\limsup _{t\to T^-}\sup_{M}|Rm|_{g(t)} = \infty.\]
 }
 Suppose there is $r>0$, such that  $S(g(0)) \geq -\frac{1}{r^2}$ on $M$ and the $\delta_0-$almost Euclidean isoperimetric inequality holds in $B_{g(0)}(x,r)$ for all $x\in M$, then $T\geq \epsilon_0^2 r^2$.
\end{lem}
\begin{proof}
%{\blue If $T<\infty$, it is easy to see
%\[\limsup _{t\to T^-}\sup_{M}|Rm|_{g(t)} = \infty.\]
%}
If $T< (\epsilon_0 r)^2$, by Theorem \ref{thm: pseudolocality}, we have
\[|Rm|(x,t)\leq \frac{A}{t} + \frac{1}{\epsilon_0^2 r^2}\]
for all $t\in (0,T)$ and all $x\in M$, which leads to a contradiction.
\end{proof}
Another ingredient in the proof of Theorem \ref{thm: short-time existence} is the following construction of a distance-like function with controlled Laplacian. Here we state a scaling invariant version.
\begin{lem}[Schoen-Yau \cite{MR1333601}]\label{lemma: distance-like function}
Let $(M,g)$ be a complete Riemannian manifold with dimension $n$ and $Ric \geq -l$ for some $l>0$. Then there is a distance-like function $\gamma:M \to \infty$, such that
\begin{equation}\label{eqn: distance-like function}
\lambda d(x,p) \leq \gamma(x) \leq \Lambda d(x,p),
\end{equation}
\[
|\nabla \gamma|\leq C_0 \quad and \quad |\Delta \gamma|\leq C_0\sqrt{l},
\]
when $d(x,p) \geq C_0/\sqrt{l}$,
where constants $\lambda, \Lambda$, $C_0$ depends only on $n$ and a lower bound of {\blue  $Vol_gB_g(p,1/\sqrt{l}) l^{n/2}$}.
\end{lem}
\begin{rem}
It is evident from the proof in \cite{MR1333601} that Lemma \ref{lemma: distance-like function} works on $B_g(p,\rho)$ for $\rho$ large enough, as long as $B_g(p,\rho+1/\sqrt{l})$ is relatively compact in $(M,g)$, even when $(M,g)$ is incomplete.
\end{rem}
In the following, we denote the level sets of $\gamma$ by
\begin{equation}\label{definition of U_rho}
U_{\rho}=\gamma^{-1}([0,\rho))
\end{equation}
for any $\rho>0$.

In the next lemma we will construct a good conformal metric on any given level set of $\gamma$. The function $f(s)$ in the proof below has been used in the work of Hochard \cite{hochard2016short}, also implicitly in Topping's \cite{topping2010ricci}, to conformally transform a compact domain into a complete manifold with bounded curvature. Note that $f(s)$ is essentially the conformal factor for a scaled hyperbolic metric on an Euclidean ball, and recall that the Euclidean isoperimetric inequality holds on hyperbolic spaces. Therefore we would like to construct a conformal factor by combining $f(s)$ with $\gamma$, then we can verify that the scalar curvature lower bound and the almost Euclidean isoperimetric inequality are roughly preserved.

\begin{lem}\label{lem: existence of a good conformal metric}
{\blue
For any $n$ and $\delta_1$, there are constants $C(n,\delta_1)$ and $c(n,\delta_1)$ with the following properties. Let $(M,g)$ be a Riemannian manifold (not necessarily complete) such that a function $\gamma$ satisfying condition (\ref{eqn: distance-like function}) is defined on it. Let $U_\rho$ be the level sets of $\gamma$ as defined in (\ref{definition of U_rho}).
Suppose there are $k \geq 0$ and $r>0$, such that on $U_{\rho}$ we have $S\geq -k$ and the  $\delta_1-$almost Euclidean isoperimetric inequality holds in any $B(x,r)\subset U_\rho$. Suppose also that $\rho>2c(n,\delta_1)C_0r\sqrt{l+1}$ and $U_\rho$ is relatively compact in $M$. Then there is a conformal metric $h$ on $U_\rho$ with the following properties:
}

(i) $(U_\rho, h)$ is a complete Riemannian  manifold with uniformly bounded curvature, and $h\equiv g$ on $U_{\rho - c(n,\delta_1)C_0 r\sqrt{l+1}}$.

(ii) the scalar curvature of $h$ is bounded from below by
\[S_h \geq -k - C(n,\delta_1)\max\{\frac{1}{r}, \frac{1}{r^2}\} ;\]

(iii) the $2\delta_1-$almost Euclidean isoperimetric inequality holds in any geodesic ball with radius $ r/8$ with respect to $h$.
\end{lem}
%\begin{rem}
%The lower bound in (ii) depends on the Ricci curvature lower bound $Ric \geq -lg$, which comes in when we use the Laplacian bound $|\Delta \gamma| \leq C_0\sqrt{l}$.
%\end{rem}
\begin{proof}
For $0<\kappa<1$ to be determined later, define $f:[0,1) \to [0, \infty)$ by
\[
f(s)=\begin{cases}
0, & 0 \leq s\leq 1-\kappa; \\
-\ln \left(1-(\frac{s-1+\kappa}{\kappa})^2\right), & 1-\kappa < s < 1.
\end{cases}
\]
By simple computation, for any $1-\kappa < s < 1$ we have
\[0 < \frac{df}{ds} = \frac{2(s-1+\kappa)}{\kappa^2 -(s-1+\kappa)^2} \leq \frac{2\kappa}{\kappa^2 - (s-1+\kappa)^2},\]
\[ 0 < \frac{d^2 f}{ds^2} = \frac{2(\kappa^2+(s-1+\kappa)^2)}{(\kappa^2-(s-1+\kappa)^2)^2} \leq \frac{4\kappa^2}{(\kappa^2 -(s-1+\kappa)^2)^2}.\]

Let $f(x)=f(\rho^{-1}\gamma(x))$ for $x\in U_{\rho}$.
Define the conformal metric $h=e^{2f}g$. Then $h=g$ on $U_{(1-\kappa)\rho}$. The completeness of $h$ is easy to check.

The scalar curvature of $h$ is given by the well-known formula
\[ S_h = e^{-2f} \left( S_g - \frac{4(n-1)}{n-2} e^{-(n-2)f/2} \Delta_g e^{(n-2)f/2}\right) \]
for $n \geq 3$. By simple calculation we have

%\[|\nabla f| =(\rho)^{-1}f' |\nabla \gamma| \leq C f',\]
%\[\Delta f = (\rho)^{-1}f' \Delta \gamma + (\rho)^{-2}f'' |\nabla \gamma|^2,\]
%and
\[
S_h = e^{-2f} \left[ S_g -\frac{4(n-1)}{n-2}\left( (\frac{n-2}{2\rho})^2 |f'|^2|\nabla \gamma|^2 + \frac{n-2}{2\rho^2} f'' |\nabla \gamma|^2 + \frac{n-2}{2\rho} f' \Delta \gamma\right) \right].
\]
Then by Lemma \ref{lemma: distance-like function} and the fact that $f'e^{-2f} \leq 2/\kappa$, $|f'|^2e^{-2f} \leq 4/\kappa^2$ and $f''e^{-2f}\leq 4/\kappa^2$, we have
\begin{equation}\label{eqn: scalar curvature lower bound for h}
S_h \geq e^{-2f} \geq -k - \frac{4(n-1)^2C_0^2}{\rho^2 \kappa^2} - \frac{4(n-1)C_0\sqrt{l}}{\rho \kappa}.
\end{equation}
For $n=2$ the calculation is similar.

Next we need to verify that $h$ has uniformly bounded sectional curvature. By the compactness of $\bar{U_\rho}$ we can find a number $K>0$ such that $|Rm_g|\leq K$ and $|\nabla^2 \rho|\leq K $ on $\bar{U_\rho}$. Hence
\[ |\nabla f(\rho^{-1}\gamma(x))|^2 e^{-2f}\leq \frac{4C_0^2}{\kappa^2 \rho^2},\]
\[ |\nabla ^2 f(\rho^{-1} \gamma(x))| e^{-2f} \leq (\frac{C_0^2}{\kappa^2 \rho}+\frac{2K}{\kappa^3})\frac{1}{\rho}. \]
Recall the formula for sectional curvature under the conformal change is
\[K^h_{ij}=e^{-2f}(K^g_{ij} - \sum_{k\neq i, j} |\nabla_k f|^2 + \nabla_i\nabla_i f + \nabla_j\nabla_j f),\]
when calculated in an orthonormal frame.
Therefore the metric $h$ have sectional curvature uniformly bounded by a constant $C(n,K,\kappa, \rho)$.

Finally we need to estimate the isoperimetric constant for the conformal metric $h$. In this step we will determine the value of $\kappa$.

Let's fix a point $y\in U_{\rho}$, suppose $\gamma(y)=d\rho$ for some $d \in [0,1)$. Let $0<\alpha<\min\{d, 1-d\}$ be a constant whose value will be determined later. Since the function $f$ is monotonically nondecreasing, on the set {\blue $U_{(d+\alpha)\rho}\setminus U_{(d-\alpha)\rho}$ }we have
\[e^{2f(d-\alpha)}g \leq h \leq e^{2f(d+\alpha)}g.\]

 For any $0< \tau < \rho/(C_0\sqrt{l+1})$, we can always choose $\alpha$ large enough such that $B_h(y, \tau) \subset B_g(y, \alpha \rho /(C_0\sqrt{l+1}))$, which is a subset of {\blue $U_{(d+\alpha)\rho}\setminus U_{(d-\alpha)\rho}$} due to the gradient bound $|\nabla \gamma| \leq C_0$. To apply the isoperimetric inequality for $g$, we would like to choose $\alpha$ small enough such that $B_g(y, \alpha \rho /(C_0\sqrt{l+1})) \subset B_g(y,r)$. {\blue Therefore we would like to show the existence of a possible $\alpha$, and in the process we will fix a value of $\tau$.}

For any $x\in B_h(y,\tau)$ we have $d_g(x,y) \leq \tau e^{-f(d-\alpha)}$. So it is sufficient to choose $\alpha$ such that
\[\tau = \frac{\alpha e^{f(d-\alpha)} \rho }{C_0\sqrt{l+1}}\]
{\blue to guarantee the inclusion $B_h(y, \tau) \subset B_g(y, \alpha \rho /(C_0\sqrt{l+1}))$.}
More precisely, let's define
\[\phi(\alpha) := \alpha e^{f(d-\alpha)}=
\begin{cases}
\frac{\alpha}{1-\left(\frac{d-1+\kappa-\alpha}{\kappa}\right)^2}, & \alpha < d- 1+\kappa; \\
\alpha, & \alpha \geq d-1+\kappa.
\end{cases}\]
It turns out that $\phi(\alpha)$ is a nondecreasing function with range $[0,1)$. For simplicity let's denote $\tilde{\tau} = C_0 \sqrt{l+1} \tau/ \rho$. by direct computation
\[\alpha_{\tilde{\tau}} := \phi^{-1}(\tilde{\tau})=
\begin{cases}
\frac{-(\kappa^2 -2\tilde{\tau} \beta) + \sqrt{(\kappa^2 -2 \tilde{\tau} \beta)^2 + 4\tilde{\tau}^2 (\kappa^2 -\beta^2 )}}{2\tilde{\tau}}, & \tilde{\tau} < \beta; \\
\tilde{\tau}, & \tilde{\tau} \geq \beta;
\end{cases}
\]
where
\[\beta=d-1+\kappa.\]
To simplify, let $\tilde{\tau} = a \kappa$, the value of $a$ will be determined later. Then $\alpha_{\tilde{\tau}}$ can be written as
\[\alpha_{\tilde{\tau}} =
\begin{cases}
\frac{2a(\kappa^2 -\beta^2)}{(\kappa-2a\beta) + \sqrt{(\kappa -2a\beta)^2+ 4a^2(\kappa^2-\beta^2)}}, & \alpha_{\tilde{\tau}} < \beta \\
a \kappa, & \alpha_{\tilde{\tau}} \geq \beta.
\end{cases}
\]
%When $\beta \leq 0$ clearly we have $\alpha_\gamma =a \kappa$.

Let's take
\[0< a < \frac{1}{4}.\]
When $\alpha_{\tilde{\tau}} <\beta$, we have
\[\alpha_{\tilde{\tau}} < \frac{2a(\kappa^2 - \beta^2)}{\kappa}<4a(\kappa-\beta). \]
If we further take $a< \frac{1}{8}$, then
\[\alpha_{\tilde{\tau}} + \beta < \frac{\kappa + \beta}{2}.\]

Claim: For any $\epsilon>0$, there is an $a$ depending only on $\epsilon$ (and independent of $d$), such that
\[|f(d+\alpha_{\tilde{\tau}}) - f(d-\alpha_{\tilde{\tau}})| <\epsilon.\]
When $\alpha_{\tilde{\tau}} < \beta$, by the definition of $f$ we have
\[e^{f(d+\alpha_{\tilde{\tau}})-f(d-\alpha_{\tilde{\tau}})} = \frac{\kappa^2 -(\beta-\alpha_{\tilde{\tau}})^2}{\kappa^2 - (\beta+\alpha_{\tilde{\tau}})^2} = 1+ \frac{4\alpha_{\tilde{\tau}} \beta}{\kappa^2 -(\beta + \alpha_{\tilde{\tau}})^2};\]
\[e^{f(d-\alpha_{\tilde{\tau}})-f(d+\alpha_{\tilde{\tau}})} = \frac{\kappa^2 -(\beta+\alpha_{\tilde{\tau}})^2}{\kappa^2 - (\beta-\alpha_{\tilde{\tau}})^2} = 1- \frac{4\alpha_{\tilde{\tau}} \beta}{\kappa^2 -(\beta - \alpha_{\tilde{\tau}})^2};\]
Using $\alpha_{\tilde{\tau}} < 4a(\kappa -\beta)$ and $\beta+\alpha_{\tilde{\tau}} < \frac{\kappa + \beta}{2}$, we can estimate
\[0<\frac{4\alpha_{\tilde{\tau}} \beta}{\kappa^2 -(\beta + \alpha_{\tilde{\tau}})^2} < 32 a;\]
\[0<\frac{4\alpha_{\tilde{\tau}} \beta}{\kappa^2 -(\beta - \alpha_{\tilde{\tau}})^2} < 16a.\]
When $\alpha_{\tilde{\tau}} \geq \beta$ (equivalently $d-\alpha_{\tilde{\tau}} \leq 1-\kappa$), we have
\[
1 \leq e^{f(d+\alpha_{\tilde{\tau}})-f(d-\alpha_{\tilde{\tau}})} \leq \frac{\kappa^2}{\kappa^2 - (\beta+\alpha_{\tilde{\tau}})^2}< \frac{1}{(1-4a^2)}.
\]
\[
1\geq e^{f(d-\alpha_{\tilde{\tau}})-f(d+\alpha_{\tilde{\tau}})} \geq \frac{\kappa^2 - (\beta+\alpha_{\tilde{\tau}})^2}{\kappa^2}> 1-\frac{4\alpha_{\tilde{\tau}}^2}{\kappa^2}=1-4a^2.
\]
Then it is clear that we can take $a$ small enough and only depending on $\epsilon$, thus the claim is proved.

Since $\alpha_{\tilde{\tau}} < 4a \kappa$ by the above analysis, we can take
\[\kappa = \frac{C_0r\sqrt{l+1}}{8 a\rho}\]
when $\rho$ is large enough (so that $\kappa < 1$), then $B_h(y, \tau) \subset B_g(y, r)$, {\blue this choice of $\kappa$ is equivalent to taking $\tau = \frac{r}{8}$}. Let $\Omega$ be a connected domain with smooth boundary in $B_h(y,\tau)$,
\[Area_h(\partial \Omega)\geq e^{(n-1)f(d-\alpha_{\tilde{\tau}})} Area_g (\partial \Omega),\]
\[Vol_h (\Omega) \leq e^{nf(d+\alpha_{\tilde{\tau}})}Vol_g (\Omega).\]
By the isoperimetric assumption in the lemma,
\[(Vol_h(\Omega))^{n-1}\leq e^{n(n-1)(f(d+\alpha_{\tilde{\tau}})-f(d-\alpha_{\tilde{\tau}}))} (1-\delta_1)I_n (Area_h(\partial \Omega))^{n}.\]
Therefore we can find a number $a$ depending on $n, \delta_1$, and determines the above constant $\kappa$, such that the almost Euclidean isoperimetric inequality
\begin{equation}\label{eqn: isoperimetric inequality for h}
(Vol_h(\Omega))^{n-1}\leq (1-2\delta_1)I_n (Area_h(\partial \Omega))^{n}
\end{equation}
holds for any regular domain $\Omega$ in $B_h(y,\tau)$. 

Now we have verified all desired properties. Let's point out that $f(s)$ is not smooth at $s=1-\kappa$. However, from the proof we can see that a smooth approximation of $f$ will work as long as it is sufficiently close to $f$ in $C^0$ norm. This can be done by standard mollifying method.
\end{proof}

\begin{proof}[Proof of Theorem \ref{thm: short-time existence}]

By the assumptions, when $R$ is sufficiently large we have {\blue $Ric \geq -(L+1) R^2$ on $B_g(p,R)$ }. The volume ratio $Vol_gB_g(p,s)/s^n$ is controlled from below by the isoperimetric inequality when $s<r$. By Lemma \ref{lemma: distance-like function} (and the remark following it) we can construct a good distance-like function $\gamma$ on $B_g(p,R)$, with constants $\lambda, \Lambda$ and $C_0$ depending only on $n$ (by requiring WLOG $\delta_1 < 1/2$). Since $\lambda d_g(p,x)< \gamma(x) < \Lambda d_g(p,x)$, we can choose $ \rho = \lambda R$, such that the level set $U_\rho \subset B_g(p,R)$.

To use Lemma \ref{lem: existence of a good conformal metric}, we have to check that {\blue $\rho > 2c(n,\delta_1)C_0 r \sqrt{L+1}R$}. This can be done by choosing {\blue $r\leq c_1/\sqrt{L+1}$} for some constant $c_1$ small enough depending on $n$ and $\delta_1$.

Therefore we have a complete metric $h$ on $U_\rho$ with bounded curvature, by Shi \cite{shi1989deforming}, a complete solution {\blue $g_\rho(t)$} with bounded curvature exists. Then we can use Lemma \ref{lem:lifespan via pseudolocality} to get a lower bound of the lifespan of Shi's solution, which is independent of $R$. 
{\blue Moreover, these solutions satisfy a curvature bound in the form
\[\sup_{U_\rho}|Rm|_{g_\rho(t)}\leq \frac{A}{t} + \frac{1}{\epsilon_0(n, A)^2 r^2}.\]
}
Therefore we can take $R \to \infty$, and the solutions converge subsequentially to a complete solution {\blue $g(t)$} on $M$. The initial metric of this limit solution is $g$ since $h=g$ on $B_g(p,2^{-1}\Lambda^{-1}\lambda R)$ for each $R$ large enough. The choice of {\blue $\bar{r}=\min\{r, \sqrt{1/(L+1)k}\}$} can be justified by scaling arguments.

Since the curvature is not bounded, this convergence does not follow directly from the compactness of Ricci flows. However, we can apply the interior curvature estimate of B.L.Chen (\cite{chen2007strong} Theorem 3.1), and the modified Shi's estimates (\cite{lu2005uniqueness} Theorem 11) to guarantee that the convergence is smooth and uniform on any compact set. 

{\blue To show the completeness, note that we can apply Lemma \ref{lem: distance lower bound} to each solution on $U_\rho$ for $0< t < \epsilon_0^2 r^2 A$, and the conclusion of Lemma \ref{lem: distance lower bound} passes to the limit solution to show
\[d_{g(t)}(x,y) \geq d_{g(0)}(x,y) - 4(n-1)(2A+1)\sqrt{t},\]
for any $x, y \in M$. Then the completeness of $g(t)$ follows since $g(0)$ is complete. The completeness of $g(t)$ for larger $t$ follows from well-known results since $g(t)$ has bounded curvature for $t>0$.
}

The non-collapsing claim comes from applying Lemma \ref{lem: Perelman's short-time noncollapsing} to every bounded curvature solution in the converging sequence.
\end{proof}
{
\blue
To prove Corollary \ref{cor: Ricci and injectivity radius bounded below}, we need the following lemma, which is a special case ($K=0$) of the more general Theorem 1 in \cite{cavalletti2017almost}.
\begin{lem}[Cavalletti - Mondino]\label{lem: almost Euclidean volume}
For any dimension $n\geq 2$, there are constants $\bar{\epsilon}(n)$, $\bar{\eta}(n)$, $\bar{\delta}(n)$ and $C(n)$ satisfying the following property. Let $(M,g)$ be a smooth $n$-dimensional Riemannian manifold and let $\bar{x}\in M$, suppose $B_g(\bar{x}, 1)$ is relatively compact in $M$ and for some $0< \epsilon \leq \bar{\epsilon}(n)$ and $0 < \eta \leq \bar{\eta}(n)$ it holds:

(1) $Vol_g B_g(\bar{x}, 1) \geq (1-\eta) \omega_n$,

(2) $Ric \geq - \epsilon g$ on $B_g(\bar{x}, 1)$.

\noindent
Then for any $0 < \delta < \bar{\delta}(n)$, the almost Euclidean isoperimetric inequality
\[
(Area_g(\partial \Omega))^n \geq (1-C(n)(\delta+\epsilon+\eta))I_n (Vol_g(\Omega))^{n-1}
\]
holds for any $\Omega \subset B_g(\bar{x}, \delta)$.
\end{lem}

%\begin{cor}
%For any dimenion $n\geq 2$, any constants $A>0$, $\epsilon > 0$, there exist constants $\eta(n, A, \epsilon)$ and $\tau(n, A, \epsilon)$, such that the following holds. Let $(M,g)$ be a complete $n$-dimension smooth Riemannian manifold with the following conditions:

%(1) there is a constant $l \geq 0$, such that $Ric \geq -l g$,

%(2) $Vol_g B_g(x, r) \geq (1-\eta) \omega_n r^n$ for any $r \in [0, 1]$.

%\noindent 
%Then the Ricci flow has a short-time solution $(M,g(t))$ with $g(0)=g$, satisfying properties:

%(i) $|Rm|(t) \leq \frac{A}{t} + \frac{l}{\tau}$;

%(ii) $Vol_{g(t)}B_{g(t)}(x, \sqrt{t}) \geq (1-\epsilon)\omega_n(\sqrt{t})^n$;

%\noindent where $t\in [0,\tau/l]$.
%\end{cor}
\begin{proof}[Proof of Corollary \ref{cor: Ricci and injectivity radius bounded below}]
Conditions (i) and (ii) in Theorem \ref{thm: short-time existence} are clearly implied by the Ricci lower bound $Ric(g) \geq -lg$.
To check the condition (iii) in Theorem \ref{thm: short-time existence},
first rescale the metric $\tilde{g} = Q g$, then $Ric(\tilde {g}) \geq  - \frac{l}{Q} \tilde{g}$. Choose $Q$ such that 
\[
\frac{l}{Q} = \min\{ \frac{1}{3}  \frac{\delta_1}{C(n)}, \bar{\epsilon}(n)\},
\]
where (and in the rest of the proof) $\delta_1$ is taken to be the minimum of the constant from Theorem \ref{thm: short-time existence} and that from Proposition \ref{prop: short-time noncollapsing}, $C(n)$ and $\bar{\epsilon}(n)$ are the constants from Lemma \ref{lem: almost Euclidean volume}. Note that the constant $Q$ defined above depends on $n, A$ and $\epsilon$.
Similarly define $\eta = \min\{ \frac{1}{3}  \frac{\delta_1}{C(n)}, \bar{\eta}(n)\}$ and $r = \min\{ \frac{1}{3}  \frac{\delta_1}{C(n)}, \bar{\delta}(n)\}$, where $\bar{\eta}(n)$ and $\bar{\delta}(n)$ are also from Lemma \ref{lem: almost Euclidean volume}, note that both $\eta$ and $r$ depends on $n$, $A$ and $\epsilon$. 

By the scaling invariance of volume ratio we see that 
$Vol_{\tilde{g}} B_{\tilde{g}}(x, 1) \geq (1-\eta) \omega_n$, for all $x\in M$. Hence we can apply Lemma \ref{lem: almost Euclidean volume} to conclude that the $\delta_1$-almost isoperimetric inequality holds on $B_{\tilde{g}}(x, r)$ for all $x\in M$. Then by the scaling invariance of the isoperimetric inequality, we know that the $\delta_1$-almost isoperimetric inequality holds on $B_{g}(x, r/\sqrt{Q})$ for all $x\in M$ with respect to the unscaled metric $g$.

Then Theorem \ref{thm: short-time existence} implies that there exists a complete solution of the Ricci flow $g(t)$, $t \in[0, \tau/l]$, with $g(0)=g$, $\tau$ is a constant depending on $n, A, \epsilon$. Moreover the solution satisfies
\[
|Rm|(g(t)) \leq \frac{A}{t} + \frac{l}{T}.
\]
And once we have the Ricci flow solution provided by Theorem \ref{thm: short-time existence}, we see from the above discussion that the assumptions of Proposition \ref{prop: short-time noncollapsing} are satisfied, hence we have (ii).
\end{proof}

}

%Now let's restate Corollary  \ref{cor: Ricci and injectivity radius bounded below} more precisely.
%\begin{cor}
%Let $(M^n, g)$ be a Riemannian manifold, suppose $Ric(x) \geq -l$ and $inj_x \geq \iota >0$ for all $x\in M$, then for any $\epsilon>0$, the Ricci flow has a short-time solution $(M,g(t))$ with $g(0)=g$, satisfying properties:

%(i) $|Rm|(t) \leq \frac{1}{t}$;

%(ii) $Vol_{g(t)}B_{g(t)}(x, \sqrt{t}) \geq (1-\epsilon)\omega_n(\sqrt{t})^n$;

%\noindent where $t\in [0,T]$, $T$ only depends on $n, l, \iota$ ande $\epsilon$.

%\end{cor}
{\blue

\begin{proof}[Proof of Remark \ref{rem: injectivity radius lower bound}]
Suppose on a complete smooth Riemannian manifold $(M,g)$ we have $Ric \geq -l$ and the injectivity radius $inj_x > \iota > 0$ for any $x\in M$. By \cite{anderson1992c}, for any $x>0$ and $\epsilon>0$, $(M,g)$ has $W^{1,p}$-harmonic radius centered at $x$ bounded from below by a uniform constant $r_h(n,l,\iota, \epsilon)$, we refer to \cite{anderson1992c} for the definition of harmonic radius. In particular, there exists coordinate functions $(x_1, x_2, ..., x_n)$ on $B_g(x, r_h)$, such that in these coordinates, the metric $g$ satisfies
\[(1-\epsilon) \delta_{ij} \leq g_{ij} \leq (1+\epsilon)\delta_{ij}.\]
Hence for any $\Omega \subset B_g(x, r_h)$, we have 
\[(Area_g(\partial \Omega))^n \geq (\frac{1-\epsilon}{1+\epsilon})^{n(n-1)}I_n (Vol_g(\Omega))^{n-1}.\]
Clearly the $\delta_1-$almost Euclidean isoperimetric inequality can be verified in $B(x,r_h)$ when we choose $\epsilon$ small enough. Since $r_h$ is independent of $x$, the assumptions of Theorem \ref{thm: short-time existence} and Proposition \ref{prop: short-time noncollapsing} are satisfied.
\end{proof}

}

\section{Short-time analysis of the solutions}

Although the Ricci flow solutions provided by Theorem \ref{thm: short-time existence} are smoothly continuous to the initial data, the modulus of continuity a-priori depends on the initial sectional curvature bounds which are not assumed. In order to guarantee these solutions have some compactness property in the sense of \cite{hamilton1995compactness}, we would need uniform non-collapsing estimates depending only on the assumed conditions. In fact, a non-collapsing result is directly implied by the proof of the pseudolocality theorem as pointed out by Perelman in the compact case \cite{perelman2002entropy}, see also \cite{kleiner2008notes}. The complete noncompact case can be proved by the same argument with the help of heat kernel estimates established in \cite{chau2779131pseudolocality}, we sketch a proof following \cite{kleiner2008notes}.
\begin{lem}[Perelman]\label{lem: Perelman's short-time noncollapsing}
There is a constant $\kappa(n)$ depending only on the dimension $n$, such that under the same assumptions of Theorem \ref{thm: pseudolocality}, we have
\[Vol_{g(t)}B_{g(t)}(x,\sqrt{t})\geq \kappa(n)t^\frac{n}{2}\]
 for all $t\in(0, \epsilon_0^2 r^2]$ and $x\in B_{g(t)}(x_0, \epsilon_0 r)$, {\blue where $\epsilon_0$ is the same constant as in Theorem \ref{thm: pseudolocality}}.
\end{lem}
\begin{proof}[Sketch of proof]
For simplicity  let's take $A=1$ in the pseudolocality theorem and assume WLOG that $|Rm|(t) \leq \frac{1}{t}$.

For any $\epsilon>0$, let $H(y,t)$ be the conjugate heat kernel centered at $(x,2\epsilon^2)$, i.e.
 \[(\frac{\partial}{\partial t} + \Delta_{g(t)} -S(t))H =0, \quad t \in (0, 2\epsilon^2),\]
\[ \lim_{t \to 2\epsilon^2} H = \delta_x.\]
Define the function $f$ by the relation $H = (4\pi (2\epsilon^2-t))^{-n/2} e^{-f}$. Choose a time dependent cut-off function $h$ as in \cite{perelman2002entropy}, {\blue which is defined as $h(y,t) = \phi(\frac{d(y,t)+ 10n \sqrt{t}}{40n \epsilon})$, where $d(y,t)$ is the distance function to the point $x$, $\phi$ is a nonincreasing function on $\mathbb{R}$ such that $\phi=1$ on $(-\infty, 1]$, $\phi = 0$ on $[2, \infty)$, $\phi^{\prime\prime} \geq - 10 \phi$ and $(\phi^\prime)^2 \leq 10 \phi$. Note that when $t < 2\epsilon^2$, $40n \epsilon < d(y,t)+10n \sqrt{t} < 80n \epsilon$ implies $20n \epsilon < d(y,t) < 80n \epsilon$, so $h(y,t) =1$ on $B_{g(t)}(x, 20n \epsilon)$ and is compactly supported on $B_{g(t)}(x, 80n \epsilon)$. Then we can apply Lemma 8.3 of \cite{perelman2002entropy} with $r_0=\sqrt{t}$ to derive
\[
(\frac{\partial}{\partial t} - \Delta) d \geq - \frac{5n}{\sqrt{t}},
\]
on the domain where $h$ is nonconstant. 
Thus we can calculate that $h$ satisfies $(\frac{\partial }{\partial t} - \Delta_{g(t)}) h \leq \frac{1}{160n^2\epsilon^2} h$.
}
Since $\lim_{t\to 2\epsilon^2} \int hH =1$, by examining $\frac{\partial}{\partial t}\int hH$ we can derive
\[ \int h H d\mu( \epsilon^2) \geq c(n).\]
Since $h$ is compactly supported in{\blue $B_{g(\epsilon^2)}(x, 80n\epsilon)$}, we have
\[
c(n)\leq C(n)\frac{Vol_{g(\epsilon^2)}B_{g(\epsilon^2)}(x,\epsilon)}{\epsilon^n} e^{-f_m},\]
where {\blue $f_m = \inf_{B_{g(\epsilon^2)}(x, 80n\epsilon)} f(\epsilon^2)$}, and we have used volume comparison theorem and the curvature bound implicitly. Clearly $f_m\to -\infty$ when the volume ratio goes to $0$.

By the Li-Yau type Harnack estimate for $H$ in \cite{chau2779131pseudolocality}, and the curvature bound $|Rm|(t) \leq \frac{2}{\epsilon^2}$ for $t\in[\epsilon^2/2, \epsilon^2]$, we can show that
\[ \int (\frac{\epsilon^2}{2} |\nabla f|^2 + f -n) h H d\mu(\frac{\epsilon^2}{2}) \leq C(n) + f_m,\]
which will be strictly negative when $f_m < -C(n)-1$. This shows after integration by parts that the local entropy at time $\frac{\epsilon^2}{2}$ is bounded from above by a negative number. Then a contradiction follows the same argument as in the proof of the pseudolocality theorem.

\end{proof}

If we assume that the initial Ricci curvature is bounded from below, we can prove the following variant by a different method, where the non-collapsing constant can be made arbitrarily close to the Euclidean value.
\begin{prop}\label{prop: short-time noncollapsing}
For any $n$, $k$, $Q$ and $\epsilon> 0$, there are constants $\delta_1$ and $\tau$ depending only on $n,k,Q$ and $\epsilon$, with the following property.
Let $g(t)$ be a complete Ricci flow solution on $M^n \times [0, \tau r^2]$ obtained from Theorem \ref{thm: short-time existence}, where $r>0$.
Suppose

(i) $Ric(g(0)) \geq -\frac{k}{r^2}$ on $B_{g(0)}(p, r)$,

(ii) The $\delta_1-$almost Euclidean isoperimetric inequality holds in $B_{g(0)}(p,r)$.

\noindent Then
\[Vol_{g(t)} B_{g(t)}(p, Q\sqrt{t} ) \geq (1-\epsilon)\omega_n Q^nt^\frac{n}{2}\]
for all $0< t \leq \tau r^2 $, where $\omega_n$ is the volume of the unit ball in $\mathbb{R}^n$.
\end{prop}
Stronger noncollapsing result can be proved when there is better control of the Ricci lower bound along the flow. In fact, when the dimension $n=3$, with the help of the Hamilton-Ivey  estimate, it has been proved in \cite{hochard2016short} that the volume is non-collapsed under a uniform scale.

In the rest of this section we will prove Proposition \ref{prop: short-time noncollapsing}, and analyse distance distortion under the Ricci flow solutions of Theorem \ref{thm: short-time existence}.

Since the solutions are obtained  as a limit of bounded curvature solutions, Proposition \ref{prop: short-time noncollapsing} follows from the lemma below. The proof explores a similar idea as in the metric lemma of \cite{hochard2016short} which depends on the continuity of measure under Gromov-Hausdorff convergence proved by Cheeger and Colding \cite{cheeger1997structure}, and we also need Perelman's pseudolocality.

\begin{lem}\label{lem: short-time noncollapsing}
For any $n$, $k$, $Q$ and $\epsilon> 0$, there are constants $\delta_1$ and $\tau$ depending only on $n,k,Q$ and $\epsilon$, with the following property.
Let $g(t)$ a complete Ricci flow solution with bounded curvature on $M^n \times [0, \tau]$.
Suppose

(i) $Ric(g(0)) \geq -k$ on $B_{g(0)}(p, 1)$,

(ii) The $\delta_1-$almost Euclidean isoperimetric inequality holds in $B_{g(0)}(p,1)$.

Then
\[Vol_{g(t)} B_{g(t)}(p, Q\sqrt{t} ) \geq (1-\epsilon)\omega_n Q^nt^\frac{n}{2}\]
for all $0< t \leq \tau$, where $\omega_n$ is the volume of the unit ball in $\mathbb{R}^n$.
\end{lem}
\begin{proof}
We only need to show the proof for $Q=1$.
Suppose the claim is not true, then for some $n, k, \epsilon$, there is a sequence of $\delta_i \to 0$, a sequence of $t_i \to 0$, and a sequence of pointed Ricci flow solutions $(M_i, g_i(t), p_i), i=1,2,...$  with $Ric(g_i(0))\geq -kg_i(0)$ and the $\delta_i-$almost Euclidean isoperimetric inequality holds on $B_{g_i(0)}(p_i,1)$. In addition, since these solutions are smooth, we can choose $t_i$ such that
\[Vol_{g_i(t)} B_{g_i(t)}(p_i, \sqrt{t}) > (1-\epsilon)\omega_n t^\frac{n}{2}\]
for any $0 < t < t_i $, while
\[Vol_{g_i(t_i)}B_{g(t_i)}(p_i, \sqrt{t_i}) = (1-\epsilon) \omega_n t_i^\frac{n}{2}.\]

Let $A_j \to 0$ be a sequence of positive numbers decreasing to $0$. For each $A_j$, let $\bar{\delta}_j>0$ and $\epsilon_j >0$ be the claimed constants in Theorem \ref{thm: pseudolocality}. Since $\delta_i \to 0$, for each $j$, we have $\delta_i< \bar{\delta}_j$ when $i$ is large enough. Thus Theorem \ref{thm: pseudolocality} yields curvature estimates
\[|Rm|(g_i(x,t)) \leq \frac{A_j}{t} + \frac{1}{(\epsilon_j r_j)^2}, \quad x \in B_{g_i(t)}(p_i, \frac{1}{2}),\quad t\in (0, (\epsilon_j r_j)^2 ],\]
when $i$ is large enough. Since $t_i \to 0$, we can find a subsequence of $\{t_i\}$, which we denote as $\{t_j\}$ for simplicity, such that $t_j < A_j (\epsilon_j r_j)^2$, hence
\[|Rm|(g_j(x,t)) \leq \frac{2A_j}{t}, \quad x \in B_{g_j(t)}(p_j, \frac{1}{2}), \quad t\in ( 0, t_j ], \quad j=1,2,...\]
Now we dilate this sequence of solutions to normalize their existence time intervals on which we have the above curvature control. Define
\[\tilde{g}_j(t) = t_j^{-1} g(t_j t), \quad t\in [0,1], \quad j=1,2,...\]
We have
\[|Rm|(\tilde{g}_j(x,t)) \leq \frac{2A_j}{t}, \quad x \in B_{\tilde{g}_j(t)}(p_j, \frac{1}{2\sqrt{t_j}}),\]
\[Vol_{\tilde{g}_j(t)}B_{\tilde{g}_j(t)}(p_j, \sqrt{t}) > (1-\epsilon)\omega_n t^\frac{n}{2}\]
for $t\in (0,1)$, and
%\[|Rm|(\tilde{g}_j(x,1)) \leq 2A_j, \quad x \in B_{\tilde{g}_j(1)}(p_j, \frac{1/4 - r_j}{\sqrt{t_j}}),\]
\[Vol_{\tilde{g}_j(1)}B_{\tilde{g}_j(t)}(p_j, 1) = (1-\epsilon)\omega_n.\]

A subsequence of the solutions $\{ B_{\tilde{g}_j(t)}(p_j, \frac{1}{2\sqrt{t_j}}), \tilde{g}_j(t), t\in(0,1], p_j \}$ converge smoothly in the pointed Cheeger-Gromov sense to a Ricci flow solution $(M_{\infty}, \tilde{g}_{\infty}, t\in(0,1], p_{\infty})$. Moveover, the convergence is uniform on $\Omega \times[s,1]$ for any compact domain $\Omega \subset M_\infty$ and $s>0$. The limit solution is complete since the radius $\frac{1}{2\sqrt{t_j}} \to \infty$ as $j\to \infty$. And the limit metric $g_\infty(t)$ is flat for any $t\in(0,1]$ since $A_j \to 0$. Therefore, each time slice of the limit solution $(M_\infty, g_\infty(t))$ is the same quotient of the Euclidean space $\mathbb{R}^n$, so we can drop the time variable and denote it as $(M_\infty, g_\infty, p_\infty)$. Clearly it has
\[Vol_{g_\infty}B_{g_\infty}(p_\infty, 1) = (1-\epsilon)\omega_n,\]
which implies that the smooth manifold $(M_\infty, g_\infty)$ is a nontrivial quotient of $\mathbb{R}^n$, so there is some constant $C_\infty$, such that
\[Vol_{g_\infty} B_{g_\infty}(p_\infty, \rho) \leq C_\infty \rho^{n-1}\]
when $\rho$ is large.

The initial data $\{B_{\tilde{g}_j(0)}(p_j, \frac{1}{2\sqrt{t_j}}), g(0), p_j \}$ converges subsequentially in the the pointed Gromov-Hausdorff sense to a metric space $(X, d_X, p_X)$. Since this sequence is not collapsed, by \cite{cheeger1997structure}, the Riemannian measure converges to the $n-$dimensional Hausdorff measure on $X$. Thus the $\delta_j-$almost Euclidean isoperimetric inequalities with $\delta_j \to 0$ imply that
\[\mathcal{H}_X^n (B_X(p_X, \rho)) = \omega_n \rho^n\]
for all $\rho>0$.

For any $\rho >0$, we can apply Lemma \ref{lem: distance lower bound} and Lemma \ref{lem: distance upper bound} when $j$ is large enough to obtain
\begin{equation}\label{eqn: distance distortion control}
 c_1(n)d_{\tilde{g}_j(t)}(x,y)-C_2(n) \leq d_{\tilde{g}_j(0)} (x,y) \leq d_{\tilde{g}_j(t)}(x,y) + 2(n-1)(2A_j + 1) \sqrt{t}
\end{equation}
for $x,y\in B_{\tilde{g}_j(t)}(p_j, \rho)$ and $t\in[0,1]$.

\emph{Claim}: There exists a surjective map $f: M_\infty \to X$ with $f(p_\infty)=p_X$, which satisfies
\[c_1d_{g_\infty} (y_1, y_2) -C_2 \leq  d_X(f(y_1),f(y_2)) \leq d_{g_\infty}(y_1,y_2) \quad all \quad y_1,y_2  \in M_\infty.\]

\emph{Proof of Claim}: For simplicity, we define $X_j$ to be $B_{\tilde{g}_j(0)}(p_j, \frac{1/4 - r_j}{\sqrt{t_j}})$ equipped with the metric $\tilde{g}_j(0)$; and define $Y_j(t)$ to be $B_{\tilde{g}_j(t)}(p_j, \frac{1/4 - r_j}{\sqrt{t_j}})$ with $\tilde{g}_j(t)$. And we will omit the isometries involved in the pointed convergence.

Let $\mathcal{D}=\{y_i\}$ be a countable dense subset of $M_\infty$. Let's first fix a $t\in (0,1]$. For each $y_i$, let $\{y_i^j\}$ be its lifts to the converging sequence $Y_j(t)$.  Equation (\ref{eqn: distance distortion control}) guarantees that $\{y_i^j\}$ is a bounded sequence in $X_j$ when $j$ is large enough, hence we can pass to a subsequence of $\{X_j\}$ (which clearly still converges to the same $X$), so that $\{y_i^j\}$ converge to some point in $x_i \in X$. We define $f(y_i)=x_i$. By passing to a diagonal sequence we can define $f$ on $\mathcal{D}$. It's clear that we can define $f(p_\infty) =p_X$.

To make $f$ surjective, let $\mathcal{E}=\{x_k\}$ be a countable dense subset of $X$, we will construct $f$ so that $\mathcal{E}$ is in the image. For each $x_k$, let $\{x_k^j\}$ be its lifts in the converging sequence $X_j$. (\ref{eqn: distance distortion control}) guarantees that $\{x_k^j\}$ is a bounded sequence in $Y_j(t)$ for $j$ large enough. By passing to a subsequence of $\{Y_j(t)\}$ we can have $\{x_k^j\}$ converge to some point $y_k \in M_\infty$. We add $y_k$ to $\mathcal{D}$ and define $f(y_k)=x_k$, clearly this agrees with the definition in the previous paragraph.

Equation (\ref{eqn: distance distortion control}) implies that
\[d_X(f(y_1), f(y_2)) \leq d_{g_\infty}(y_1,y_2) + 2(n-1)\sqrt{t}\quad for \quad any \quad y_1, y_2 \in \mathcal{D}.\]
Since $\{Y_j(t)\}$ converges to the same $(M_\infty, g_\infty)$ for all $t\in(0,1]$, we can let $t\to 0$ and pass to a subsequence again to define $f$, such that
\[d_X(f(y_1), f(y_2)) \leq d_{g_\infty}(y_1,y_2) \quad for \quad any \quad y_1, y_2 \in \mathcal{D}.\]
Now, $f$ is defined on a dense subset $\mathcal{D}$, it's uniformly Lipschitz continuous, and its image contains a dense subset $\mathcal{E}$, thus $f$ can be extended to a surjection from $M_\infty$ to $X$. The first inequality in (\ref{eqn: distance distortion control}) naturally passes to the limit. So the claim is proved.

\vspace{10pt}

By the definition of Hausdorff measure we have
\[\mathcal{H}_X^n(\Omega) \leq \mathcal{H}_{M_\infty}^n(f^{-1}(\Omega)), \quad \Omega \subset X.\]
Since $f^{-1}(B_X(p_X, \rho)) \subset B_{g_\infty}(p_\infty, c_1^{-1}(\rho+C_2))$, and the Hausdorff measure coincide with the Riemannian measure on a Riemmannian manifold, we have
\[\mathcal{H}_X^n(B_X(p_X, \rho)) \leq \mathcal{H}_{M_\infty}^n (f^{-1}(B_X(p_X, \rho)))\leq \mathcal{H}_{M_\infty}^n (B_{g_\infty}(p_\infty, c_1^{-1}(\rho+C_2))),\]
which, by previous analysis, implies that
\[\omega_n \rho^n \leq C_\infty c_1^{-n-1}(\rho+C_2)^{n-1}\]
for all $\rho>0$, clearly this is a contradiction.
\end{proof}
The following lemma of Perelman \cite{perelman2002entropy} is frequently used in the study of Ricci flows.
{\blue
\begin{lem}[Perelman]\label{lem: derivative of changing distance}
Let $(M^n,g(t))$, $t \in [0,T)$, be a Ricci flow solution, for any $t_0 \in [0,T)$ and $x_0, x_1 \in M$, suppose
$Ric(x,t_0) \leq (n-1)K$ for all $x\in B_{g(t_0)}(x_0, r_0) \cup B_{g(t_0)}(x_1, r_0)$ for some $K \geq 0$ and $r_0 >0$. Assume also that $M$ contains a tubular neighbourhood of the shortest geodesic joining $x_0$ and $x_1$. Then
\[ \frac{\partial}{\partial t} |_{t=t_0} d_{g(t_0)}(x_0, x_1) \geq - 2(n-1)\left(Kr_0 +\frac{1}{r_0}\right).\]
\end{lem}
When the curvature is bounded by a function $\frac{A}{t}$, the above lemma implies the following distance distortion estimate from below:
\begin{lem}\label{lem: distance lower bound}
Let $(M^n,g(t))$, $t\in[0,T)$ be a Ricci flow solution. Suppose there is a constant $A$, $|Ric|(x,t) \leq \frac{A}{t}$ for any $t\in(0,T)$ and $x\in B_{g(t)}(p, r + \sqrt{t})$. Then for any $x, y \in B_{g(t)}(p, r- 4(n-1)(A+1)\sqrt{t})$, we have
\[d_{g(t)}(x,y) \geq d_{g(0)}(x,y) - 4(n-1)(A+1)\sqrt{t}.\]
\end{lem}
}
\begin{proof}
Suppose $\rho$ is the largest number such that $B_{g(t)}(p, \rho) \subset B_{g(s)}(p,r)$ for all $s\in[0,t]$. For any $x,y \in B_{g(t)}(p,\rho)$ and any $s\in[0,t]$, we can take $r_0=\sqrt{s}$ in Lemma \ref{lem: derivative of changing distance} which yields
\[ \frac{\partial}{\partial t} |_{t=s} d_{g(t)}(x,y) \geq - \frac{2(n-1)(A+1)}{\sqrt{s}}.\]
Hence
\[ d_{g(t)}(x,y) \geq d_{g(s)}(x,y) - 4(n-1)(A+1)(\sqrt{t}-\sqrt{s}),\]
which implies $B_{g(t)}(p, r - 4(n-1)(A+1) (\sqrt{t} - \sqrt{s}) ) \subset B_{g(s)}(p,r)$ for all $s\in[0,t]$.
Therefore $\rho \geq r - 4(n-1)(A+1)\sqrt{t}$, and we have the claimed result.
\end{proof}
Our strategy to obtain a distance distortion upper estimate is through controlling volume inflation of large domains, and controlling volume of small geodesic balls from below. Since the volume element under the Ricci flow satisfies
\[\frac{\partial}{\partial t} d\mu = -S d\mu,\]
the following local estimate of scalar curvature lower bound proved by B.L.Chen \cite{chen2007strong} can help us control the volume. (Note that when the curvature is bounded by $\frac{A}{t}$, we can choose a time dependent cut-off function as in the proof of Lemma \ref{lem: Perelman's short-time noncollapsing}, then the original proof of Proposition 2.1 in \cite{chen2007strong} still works.)
\begin{lem}[B.L.Chen]\label{lem: local scalar curvature lower bound}
Let $(M^n, g(t)), t\in[0,T)$ be a Ricci flow solution. Suppose for some $r \geq 2\sqrt{T}$, $B_{g(t)}(p,2r)$ is relatively compact in $M$ for each $t$, $|Ric|(g(t)) \leq \frac{A}{t}$ on $B_{g(t)}(p,2r)$, and $S(g(0)) \geq -k$ on $B_{g(0)}(p, 2r)$,
then
\[S(x,t) \geq - \max\{ k, \frac{C(n,A)}{r^2}\},\]
where $x\in B_{g(t)}(p, r)$, $t\in[0,T)$, and $C(n,A)$ is some constant depending only on $n$ and $A$.
\end{lem}

\begin{lem}\label{lem: distance upper bound}
For any $n,k,A,v_0>0$, there is a constant $\Lambda(n,A)$, with the following property.
Let $(M^n,g(t))$, $t\in[0,T)$ be a Ricci flow solution, $r> (1+2B)\sqrt{T}$ where $B = 4(n-1)(A+1)$, and $\tilde{r} = \Lambda e^{T\max\{k,1/r^2\}}v_0^{-1}r$. Suppose $B_{g(t)}(p, 2\tilde{r}+\sqrt{t})$ is relatively compact in $M$ for each $t\in[0,T]$, and

(i) $Ric(g(0)) \geq -k g(0)$ on $B_{g(0)}(p, 2r)$,

(ii) $|Ric|(x,t) \leq \frac{A}{t}$ for all $x\in B_{g(t)}(p, 2\tilde{r})$, $t\in(0,T)$;

(iii) $Vol_{g(t)}B_{g(t)}(x, \sqrt{t}) \geq v_0 t^\frac{n}{2}$ for all $x\in B_{g(t)}(p,2\tilde{r})$ and $t\in(0,T]$.

\noindent Then for any $x,y \in B_{g(0)}(p,r)$, and for any $t\in[0,T]$, we have
\[d_{g(t)}(x,y) \leq \frac{\Lambda e^{T\max\{k,1/r^2\}}}{v_0} \sqrt{T} \quad when \quad d_{g(0)}(x,y) < 2B\sqrt{T},\]
\[d_{g(t)}(x,y) \leq \frac{\Lambda e^{T\max\{k,1/r^2\}}}{v_0} d_{g(0)}(x,y) \quad when \quad d_{g(0)}(x,y) \geq 2B\sqrt{T}.\]

\end{lem}
\begin{proof}
Let's define $\tilde{r}(t) = \Lambda e^{t\max\{k,1/r^2\}}v_0^{-1}r$ for positive $t$. {\blue Note that $\Lambda$ is undefined yet, but for the moment we can choose it to be any large number, for example $100\omega_n$, and later choose the value of $\Lambda$ to be greater than $100\omega_n$ and subject to other requirements. Here it might be helpful to note that we naturally have $v_0 \leq \omega_n$, the volume of a unit $n$-dimensional Euclidean ball.}

{\blue
Since $\tilde{r}(0) = \Lambda v_0^{-1} r > r$ as we choose $\Lambda$ large enough, we have $B_{g(0)}(p,r) \subset B_{g(0)}(p,\tilde{r}(0))$. By the smoothness of the flow, if $\tau>0$ is sufficiently small then we have 
$B_{g(0)}(p, \tilde{r}(0)) \subset B_{g(t)}(p, 2\tilde{r}(0))$ for all $0\leq t < \tau$. Note that $\tilde{r}(t)$ is an increasing function, so we have $B_{g(0)}(p, r) \subset B_{g(t)}(p, 2\tilde{r}(\tau))$ for all $0\leq t < \tau$.
So there exists a number $0 < T_1 \leq T$, such that $T_1$ is the largest number not bigger than $T$ satisfying $B_{g(0)}(p,r) \subset B_{g(t)}(p,2\tilde{r}(T_1))$ for all $t\in[0,T_1]$.
}

We first examine the dilation of geodesic balls with the fixed radius $B\sqrt{T_1}$.
For any $x \in B_{g(0)}(p,r - (1+B)\sqrt{T_1})$, suppose the diameter of $B_{g(0)}(x, B\sqrt{T_1})$ with respect to $g(T_1)$ is $D>0$, then the $\sqrt{T_1}-$neighborhood of $B_{g(0)}(x, B\sqrt{T_1})$ with respect to $g(T_1)$ contains at least $[\frac{D}{2\sqrt{T_1}}]$ disjoint geodesic balls with radius $\sqrt{T_1}$ with respect to $g(T_1)$; on the other hand this neighborhood is contained in $B_{g(0)}(x, (1+2B)\sqrt{T_1})$ by Lemma \ref{lem: distance lower bound}. So by assumption (iii) we have
\[Vol_{g(T_1)} B_{g(0)}(x, (1+2B)\sqrt{T_1}) \geq [\frac{D}{2\sqrt{T_1}}] v_0 T_1^\frac{n}{2},\]
by Lemma \ref{lem: local scalar curvature lower bound} we have
\[\frac{Vol_{g(T_1)} B_{g(0)}(x, (1+2B)\sqrt{T_1})}{Vol_{g(0)} B_{g(0)}(x, (1+2B)\sqrt{T_1})} \leq e^{T_1\max\{(n-1)k, C(n,A)/r^2 \}}.
\]
Hence
\[D \leq \frac{C(n,A)e^{T_1\max\{k,1/r^2\}}}{v_0} \sqrt{T_1},\]
where we have used the Bishop-Gromov volume comparison theorem.

Now for any $x,y \in B_{g(0)}(p,r)$. Suppose $d_{g(0)}(x,y) = l$, then the shortest geodesic connecting $x$ and $y$ can be covered by $[\frac{l}{2B\sqrt{T_1}}]+1$ geodesic balls with radius $B\sqrt{T_1}$, all with respect to $g(0)$. Then the above analysis shows that
\[d_{g(T_1)}(x,y) \leq (\frac{l}{2B\sqrt{T_1}}+1 )\frac{C(n,A)e^{T_1\max\{k,1/r^2\}}}{v_0} \sqrt{T_1}.\]
We can apply Lemma \ref{lem: distance lower bound} on the time interval $[t,T_1]$ to compare $d_{g(t)}(x,y)$ and $d_{g(T_1)}(x,y)$. Now we can choose an appropriate $\Lambda$ {\blue (depending on $n$ and $A$) }and get the claimed result with $T_1$ instead of $T$.

If $T_1 < T$, then the maximality of $T_1$ implies that $B_{g(0)}(p,r)$ has to touch the boundary of $B_{g(T_1)}(p, 2\tilde{r}(T_1))$. However, this will not happen since the above estimate implies $B_{g(0)}(p,r) \subset B_{g(T_1)}(p, \tilde{r}(T_1))$. Therefore $T_1 = T$ and the proof is finished.

\end{proof}

\section{Applications}
In this section, let's restate (for readers' convenience) the applications in the introduction and provide their proofs.

If the optimal Euclidean isoperimetric inequality holds on the entire manifold where the Ricci curvature is nonnegative, then this manifold is isometric to $\mathbb{R}^n$. We can now relax the condition to nonnegative scalar curvature when the Ricci curvature is bounded from below by any negative sub-quadratic function.
\begin{cor}
Let $(M^n,g)$ be a complete Riemannian manifold. Suppose

 (i) $\liminf_{d(p,x) \to \infty }\frac{Ric(x)}{d(p,x)^2} \geq -L >  {\blue -\infty}$ for some fixed point $p$,

(ii) the scalar curvature $S(x) \geq 0$ for all $x\in M$,

(iii) the optimal Euclidean isoperimetric inequality holds on $M$.

Then $(M,g)$ is isometric to the Euclidean space $\mathbb{R}^n$.
\end{cor}
\begin{proof}
{\blue 
Take a decreasing sequence of positive numbers $A_i \to 0$, and an increasing  sequence of numbers $r_j \to \infty$. For each pair of $A_i$ and $r_j$, 
Theorem \ref{thm: short-time existence} provides a complete Ricci flow solution $g_{i, j}(t)$ on $M$ with
\[|Rm|(g_{i,j}(t)) \leq \frac{A_i}{t} + \frac{1}{c(n,A_i)r_j^2}, \quad t\in (0, c(n,A_i)r_j^2).\]
Note that since {\blue $k=0$} in this case, the time intervals of existence no longer depend on $k$ or $L$.

Now for each fixed $i$ and $T > 0$, we can take $j \to \infty$,  a subsequence of $g_{i,j}(t)$ will converge in the pointed Cheeger-Gromov sense to a complete smooth solution $g_{i}(t)$, $t\in[0,T] $ for the same reason as explained in the proof of Theorem \ref{thm: short-time existence}. As $T$ is arbitrary, we can use a diagonal argument to construction a limit solution with existence time interval $[0, \infty)$, for simplicity we still denote it as $g_i(t)$. Since $r_j \to \infty$,  $g_i(t)$ satisfies the curvature bound
\[|Rm|(g_i(t)) \leq \frac{A_i}{t}, \quad t \in(0,\infty).\]
Then take $i \to \infty$, a subsequence of $g_i(t)$ will converge in the pointed Cheeger-Gromov sense to a complete smooth solution $g(t)$. Since we have the curvature upper bounds
\[|Rm|(g_i(t)) \leq \frac{A_i}{t} \to 0 \quad as \quad i\to \infty \quad for \quad t>0,\]
the limit solution $g(t)$ must be flat for any $t>0$, thus the initial metric must be flat. The isoperimetric inequality implies maximal volume growth, hence the initial manifold is isometric to $\mathbb{R}^n$.

}
\end{proof}

We also prove a local result concerning the topology of almost isoperimetrically Euclidean geodesic balls with Ricci curvature bounded from below.
\begin{cor}
For any $n$ and $k$, {\blue there} exist constants $\delta$ and $\eta$ depending only on $n$ and $k$, such that if the geodesic ball $B(p,1)$ is relatively compact in a Riemannian manifold $(M,g)$, $Ric \geq -k $ on $B(p,1)$, and the $\delta-$almost Euclidean isoperimetric inequality holds in $B(p,1)$, then $B(p,\eta)$ is diffeomorphic to a Euclidean ball of dimension $n$.
\end{cor}
\begin{proof}
Suppose the claim is not true, then there is a sequence of manifolds $\{(M_i,g_i)\}$ with geodesic balls $\{B_i(p_i, 1)\}$ satisfying the assumptions with $\delta_i \to 0$, but the largest geodesic ball centered at $p_i$ that is diffeomorphic to a Euclidean ball is $B_i(p_i, \epsilon_i)$, and $\epsilon_i \to 0$ as $i \to \infty$.

Define $\tilde{g}_i=\epsilon_i^{-2}g_i$, then $B_{\tilde{g}_i}(p_i, \epsilon_i^{-1}) = B_i(p_i,1)$ has radius $\epsilon_i^{-1} \to \infty$, $Ric(\tilde{g}_i) \geq -\epsilon_i^2 k \tilde{g}_i$, and the $\delta_i-$almost isoperimetric inequality still holds.

Choose a sequence $A_j \to 0$. For each $A_j$, we can find $i$ large enough, such that we can apply Lemma \ref{lem: existence of a good conformal metric} to conformally transform $B_{\tilde{g}_i}(1, \epsilon_i^{-1})$ to a complete manifold $(M_i, h_i)$ while keeping $B_{\tilde{g}_i}(p_i, \epsilon_i^{-1}/2)$ unchanged, and we can apply Theorem \ref{thm: short-time existence} to produce a complete solution of Ricci flow $(M_i, h_i(t))$ for $t\in[0,1]$ with
\[|Rm|(h_i(t)) \leq \frac{A_j}{t}, \quad t\in(0,1].\]
Moreover, $h_i(t)$ is $\omega_n/2$-noncollapsed under the scale $Q\sqrt{t}$ near the point $p_i$ for all $t\in(0,\tau]$, where $\tau$ is from Lemma \ref{lem: short-time noncollapsing}.

We can extract a subsequence $\{M_j, h_j(\tau), p_j\}$ which converges smoothly in the pointed Cheeger-Gromov sense to $(M_\infty, h_\infty,  p_\infty)$, which is a quotient of $\mathbb{R}^n$, and
\[Vol_{g_\infty}B_{h_\infty}(p_\infty, Q) \geq \frac{\omega_n}{2} Q^n.\]
Thus there exists a number $a(n) \in (0,1)$, such that $B_{g_\infty}(p_\infty, aQ)$ is diffeomorphic to an $n-$d Euclidean ball. This topological information can be lifted to $B_{h_i(\tau)}(p_i, aQ)$ when $i$ is large enough.

On the other hand, Lemma \ref{lem: distance upper bound} implies that there is a constant $\Lambda(n)$, such that $B_{h_i(0)}(p_i, 2) \subset B_{h_i(\tau)}(p_i, \Lambda)$. Therefore, if we choose $Q\geq \Lambda/a$, then we have a contradiction with the assumption that $B_{h_i(0)}(p_i, 1+\epsilon)$ is not diffeomorphic to an $n$-d Euclidean ball for any $\epsilon>0$.

\end{proof}

{\blue
Next we discuss the application to Yau\rq{}s conjecture. The following corollary is a special case of G. Liu\rq{}s recent result \cite{liu2016gromov}. 
\begin{cor}\label{cor: nonnegative BK}
For any $n$, there is $\delta(n)>0$, such that if $(M,g)$ is a complete $2n-$dimensional K\"ahler manifold with nonnegative holomorphic bisectional curvature and maximal volume growth, and the $\delta(n)-$almost Euclidean isoperimetric inequality holds in every geodesic ball with radius $r$, then $(M,g)$ is biholomorphic to $\mathbb{C}^n$.
\end{cor}
}
\begin{proof}
It has been proved in \cite{huang2015k} that under the curvature bound $|Rm(g(t))|\leq \frac{A}{t}$ with $A$ sufficiently small, the K\"ahler condition and the nonnegativity of the holomorphic bisectional curvature are preserved. We only need to check that the maximal volume growth is preserved under the Ricci flow solutions constructed in Theorem \ref{thm: short-time existence}, which is done by the following Lemma \ref{lem: maximal volume growth kahler}. Then we have a complete K\"ahler metric with bounded nonnegative holomorphic bisectional curvature and maximal volume growth, which is biholomorphic to $\mathbb{C}^n$ by results in \cite{chau2006complex}.
\end{proof}

In \cite{huang2015k}, the authors applied the Ricci flow solution of \cite{simon2002deformation}, which is uniformly $C^0$ close to the initial metric, hence the persistence of maximal volume growth is immediate.

For the Ricci flow solutions in this article, suppose the initial Ricci curvature is bounded from below, then the  persistence of maximal volume growth can be proved using results of Section 3. However, since nonnegative holomorphic bisectional curvature implies nonnegative Ricci curvature, there is an easy proof and the non-collapsing scale can be made uniform in this situation, see the following lemma.
\begin{lem}\label{lem: maximal volume growth kahler}
Let $(M,g(t))$ be a Ricci flow solution {\blue on an $m$-dimensional manifold} with $Ric(g(t)) \geq 0$ and $|Rm|(g(t))\leq \frac{A}{t}$, and the $\delta-$almost isoperimetric inequality holds in $B_{g(0)}(x,r)$ for all $x\in M$. Suppose
\[Vol_{g(0)}B_{g(0)}(p, s) \geq v s^m\]
{\blue for $s > 0$}. Then there is a $\tau>0$ depending only on $n$ and $A$, such that
\[Vol_{g(t)} B_{g(t)} (x, s) \geq c(m)v s^m,\]
for $t\in[0,\tau]$, {\blue for $s > 0$}.
\end{lem}
\begin{proof}
The isoperimetric condition implies that (WLOG assume $\delta<<1$ )
\[Vol_{g(0)}B_{g(0)}(x,s) \geq (1-\delta)\omega_{m} s^{m}.\]
\emph{Claim:} there is a $\tau>0$ such that for any $0<t<\tau$, and for any $x$ and $0< s <r$
\[Vol_{g(t)}B_{g({\blue t})}(x,s) \geq \frac{\omega_{m}}{2} s^{m}.\]
\emph{Proof of Claim}. 
{\blue 
If the claim is not true, then there is a sequence of times $\tau_i \to 0$, a sequence of points $x_i$, and a sequence of radii $s_i < r$ such that
\[Vol_{g(\tau_i)}B_{g(\tau_i)}(x_i,s_i) < \frac{\omega_{m}}{2} s_i^{m}.\]
$Ric(g(t)) \geq 0$ implies that the volume ratio is a non-increasing function of the radius, which is also known to be continuous, hence
\[Vol_{g(\tau_i)}B_{g(\tau_i)}(x_i, r) < \frac{\omega_{m}}{2} r^{m}.\]
On the other hand we have
\[
Vol_{g(t)}B_{g(t)}(x_i, r) \to Vol_{g(0)}B_{g(0)}(x_i,r) \geq(1-\delta)\omega_{m} r^m \quad as \quad t \to 0.
\]
Thus we can choose $\tau_i$ smaller, and still denote them by $\tau_i$ for simplicity, such that
\[Vol_{g(\tau_i)}B_{g(\tau_i)}(x_i, r) = \frac{\omega_{m}}{2} r^{m}.\]
Then by the monotonicity of volume ratio we also have
\[Vol_{g(\tau_i)}B_{g(\tau_i)}(x_i, s) \geq \frac{\omega_{m}}{2} s^{m}\]
for any $0< s < r$.

}
The sequence of pointed geodesic balls with the Riemannian distance and measure $(B_{g(0)}(x_i, r), d_{g(0)}, d\mu_{g(0)})$ converges in the Gromov-Hausdorff sense to a metric measure space $(X,d_X, d\mu_X)$, where $d\mu_X$ is the Hausdorff measure. Here the convergence of the measures follows from \cite{cheeger1997structure}. On this limit space we have
\[Vol_XB_X(x_\infty, s) \geq (1-\delta)\omega_{m}.\]
On the other hand, the sequence $(B_{g(\tau_i)}(x_i, r) , d_{g(\tau_i)}, d\mu_{g(\tau_i)})$ converges to a metric measure space $(Y, d_Y, d\mu_Y)$, where $d\mu_Y$ is the Hausdorff measure. We have
\[Vol_Y B_Y(x_\infty, r) = \frac{\omega_{m}}{2} r^{m}.\]
However, by the curvature conditions and Lemma \ref{lem: distance lower bound}, we have
\[d_{g(0)}(x,y) \geq d_{g(t)}(x,y) \geq d_{g(0)}(x,y) - C(m)A\sqrt{t},\]
for any $x,y\in M$ and $0<t<T$. The continuous surjection $f:Y \to X$ constructed in the proof of Lemma \ref{lem: short-time noncollapsing} is an isometry in this situation. Hence we have the same $n$-dimensional Hausdorff measure on $X$ and $Y$, which leads to a contradiction and finishes the proof of the claim.

Now let's assume that the initial manifold has the following volume growth rate
\[Vol_{g(0)}B_{g(0)}(x, s) \geq v s^m.\]
Let $N(r,s)$ be the maximal number of disjoint $r-$geodesic balls with respect to $g(0)$ in $B_{g(0)}(x, s)$, then there are $N(r,s)$ geodesic balls with radius $2r$ that cover $B_{g(0)}(x,s)$. Hence
\[N(r,s) \geq \frac{vs^m}{\omega_{m} (2r)^{m}}.\]
Take $t< \tau$ and small enough so that $d_{g(t)} \geq d_{g(0)} - r/2$, then
\[B_{g(t)}(y, r/2) \subset B_{g(0)}(y, r), \]
hence there are at least $N(r,s)$ disjoint geodesic balls with radius $r/2$ w.r.t $g(t)$ in $B_{g(0)}(x, s) \subset B_{g(t)}(x,s)$. Therefore
\[Vol_{g(t)} B_{g(t)} (x, s) \geq N(r,s) \frac{\omega_{m}}{2} \frac{r}{2}^{m} \geq c(m)v s^m.\]
\end{proof}

\bibliographystyle{plain}
\bibliography{ref}
\end{document}